\documentclass[12pt]{amsart}
\usepackage{amsmath,amscd,amsthm,amsfonts, amssymb,amsxtra}
\usepackage[all]{xy} \SelectTips{cm}{}
\usepackage{hyperref}
  \hypersetup{colorlinks=true,citecolor=blue}
  \usepackage{graphicx}
  \usepackage[margin=1in]{geometry}
\CDat
\newtheorem{thm}{Theorem}[section]  
\newtheorem*{un-no-thm}{Theorem}
\newtheorem{cor}[thm]{Corollary}     
\newtheorem{lem}[thm]{Lemma}         
\newtheorem{prop}[thm]{Proposition}  
\newtheorem{add}[thm]{Addendum}

\newtheorem{bigthm}{Theorem}

\newtheorem{bigconj}[bigthm]{Conjecture}

\theoremstyle{definition}
\newtheorem{defn}[thm]{Definition}   

\theoremstyle{definition}

\theoremstyle{definition}
\theoremstyle{remark}
\newtheorem{rem}[thm]{Remark}

\newtheorem{hypo}[thm]{Hypothesis}
\newtheorem{notation}[thm]{Notation}

\newtheorem*{intro-rem}{Remark}
\newtheorem*{intro-rems}{Remarks}

\newtheorem{ex}[thm]{Example}

\newtheorem*{convent}{Conventions}
\DeclareMathOperator*{\holim}{holim}

\DeclareMathOperator*{\hocolim}{hocolim}

\begin{document}\
\title{Fake Wedges}
\date{\today} 
\author{John R. Klein} 
\address{Department of Mathematics, Wayne State University,
Detroit, MI 48202} 
\email{klein@math.wayne.edu} 
\author{John W. Peter} 
\address{Department of Mathematics, Utica College,
Utica, NY 13502}
\email{jwpeter@utica.edu}
\begin{abstract} A {\it fake wedge} is a diagram of spaces $K \leftarrow AÊ\to C$
whose double mapping cylinder is contractible. The terminology stems from the
 special case  $A = K\vee C$ with maps given by the projections.
In this paper, we study the homotopy
type of the moduli space $\mathcal D(K,C)$ of fake wedges on $K$ and $C$. We formulate two
conjectures concerning this moduli space and verify that these conjectures hold
after looping once. We show how embeddings of manifolds in Euclidean
space provide a wealth of examples of non-trivial fake wedges.
In an appendix, we recall discussions that the first author had with Greg Arone
and Bob Thomason in early 1995 and explain how these are related to  our conjectures.
\end{abstract}
\thanks{The first author is partially supported by the 
National Science Foundation}
\maketitle
\setlength{\parindent}{15pt}
\setlength{\parskip}{1pt plus 0pt minus 1pt}
\def\Top{\bold T\bold o \bold p}
\def\wTop{\text{\rm w}\bold T}
\def\wT{\text{\rm w}\bold T}
\def\vo{\varOmega}
\def\vs{\varSigma}
\def\smsh{\wedge}
\def\flush{\flushpar}
\def\dbslash{/\!\! /}
\def\:{\colon\!}
\def\Bbb{\mathbb}
\def\bold{\mathbf}
\def\cal{\mathcal}
\def\orb{\cal O}
\def\hoP{\text{\rm ho}P}

\setcounter{tocdepth}{1}
\tableofcontents
\addcontentsline{file}{sec_unit}{entry}

\section{Introduction \label{intro}}

There is a long history of
 constructing families of spaces which cannot be distinguished
by homology. The purpose of this paper is to study a variant
of this kind of problem for spaces that appear
to be wedges from the viewpoint of every homology theory, but
which are not the homotopy type of a wedge. 
In fact, our context is somewhat interesting in that
the spaces we consider do split as wedges after a single suspension.

More precisely, suppose $K$ and $C$ are based spaces having the
homotopy type of CW complexes.
A {\it fake wedge} on $K$ and $C$ is a based space $A$ equipped with maps $A\to K$ and
$A\to C$ such that the homotopy colimit $\hocolim (K \leftarrow A \to C)$
is contractible.

 The pair of maps in the definition can be amalgamated into a single map $A\to K\times C$ 
which we call a {\it structure map.}
We usually suppress the structure map from the notation and specify a fake wedge by its underlying space.
Two fake wedges $A$ and $A'$ on $K$ and $C$ are said to be {\it equivalent} if there is a
finite chain of weak homotopy equivalences from $A$ to $A'$ each commuting with the structure map to $K\times C$.   

For example, if $A = K \vee C$ and the maps are the projections, then the homotopy colimit is given by the wedge of reduced cones $C(K) \vee C(C)$ and is therefore contractible; we say in this case that $A$ is a {\it trivial} wedge on $K$ and $C$. More generally, we will also call a fake wedge $A$ trivial on $K$ and $C$ if is equivalent to the trivial one. 
If $A$ is a fake wedge then we can suspend
to get maps $\Sigma A\to \Sigma K$ and $\Sigma A\to \Sigma C$ to obtain a fake wedge
$\Sigma A$ on $\Sigma K$ and $\Sigma C$. However, this fake wedge is trivial, since
 we can use the comultiplication on $\Sigma A$ to add the maps and get a trivialization.
\medskip

A reservoir of non-trivial examples is  provided by embedding theory.
 Suppose that $K^n \subset S^n$ is an orientable compact codimension zero submanifold
whose boundary we will denote by $E$. 
Let $C$ be the complement of $K$.  
Then the union $K \cup_E C$ is $S^n$. 
There are two methods of manufacturing fake wedges from this situation.
The first is to remove a point $x$ from the interior of $C$. This gives a space $W$ such that
the union $K\cup_{E} W$ is $\Bbb R^n$. Therefore $E$ is a fake wedge on $K$ and $W$.
It is not necessarily trivial:

\begin{ex} Let $K = S^p \times D^q \subset S^{p+q}$ be standardly embedded.
Then $E = S^p \times S^{q-1}$ and $C = D^{p+1} \times S^{q-1}$. Moreover, $W = C-x$ has the homotopy type of 
$S^p_+\smsh S^{q-1} = S^{p+q-1} \vee S^{q-1}$ (where $S^p_+$ means $S^p$ with a disjoint basepoint). Then $E$ is a non-trivial fake wedge on $K$ and $W$,
since there are no non-trivial cup products $a\cup b$ in $K \vee W$ for $a,b$ in positive degrees.
\end{ex}

The second, and possibly more interesting way to 
manufacture examples of fake wedges from an embedding $K \subset S^n$
is to remove a point $y$ from $E$.
Let $A = E \setminus y$. Then the union $K \cup_A C$ is $\Bbb R^n$, so $A$ is a fake wedge on $K$ and $C$.
In many instances $A$ is non-trivial:

\begin{bigthm} \label{fake_wedge_from_embedding} Suppose $K\subset S^n$ is a compact tubular neighborhood of a closed connected orientable submanifold $M\subset S^n$ of codimension $\ge 3$. Let $E$ be the boundary of $K$ and let $y \in E$ be any point.
Assume $M$ is not a homology sphere. 
Then $A = E\setminus y$ is a non-trivial fake wedge on $K$ and $C$.
\end{bigthm}

\noindent (our requirement that $M$ is not a homology sphere is necessary, since the standard
embedding $S^m \subset S^n$ gives rise to a trivial wedge on $S^m$ and $S^{n-m-1}$).

\begin{rem} Although we will not pursue the matter here, it is worth mentioning that
the second construction suggests a new approach to finding codimension zero embeddings
of $K$ in $S^n$, where $K$ is a compact smooth $n$-manifold with boundary:
\begin{itemize}
\item Step 1: find the possible complements $C$ up to homotopy;
\item Step 2: find all possible fake wedges $A$ on $K$ and $C$;
\item Step 3: determine whether one can attach a cell to $A$ to build a space $E$
which maps to $K\times C$ extending the  structure map from $A$ such that the 
pairs $(K,E)$ and $(C,E)$ satisfy $n$-dimensional Poincare duality.
\item Step 4: use surgery theory to smoothen the Poincar\'e duality data in Step 3.
\end{itemize}
Step 1 concerns the existence of  Spanier-Whitehead
$n$-duals of $K$ and is a kind of desuspension question (cf.\ \cite{Klein_susp}). 
Step 2 is related to the main results of this paper. 
Step\, Ê3 is believed to be difficult and the known results are limited in scope \cite{Klein_embeddability},
\cite{Klein-Richter1}. Step 4 is achieved by application of 
the Browder-Casson-Sullivan-Wall theorem
\cite[Chap.\ 12]{Wall_book}.
\end{rem}

\subsection*{The classifying space of fake wedges} 
Define a category $D(K,C)$ as follows:
an {\it object} $A$ is a fake wedge on $K$ and $C$.
A {\it morphism} $A\to A'$ is a weak equivalence of underlying spaces which commutes with the structure
maps.
We let $\cal D(K,C)$ denote the classifying space of $D(K,C)$, i.e., the
realization of its nerve. 
It has a preferred basepoint given by the trivial wedge $K \vee C$. Then $\cal D(K,C)$ is a classifying space for fake wedges: if $B$ is a finite complex, then  up to homotopy, a map 
$B\to \cal D(K,C)$ is the same thing as specifying a fibration $E\to B$ 
together with a map $E\to K\times C$ such that for $b\in B$
the fiber $E_b \subset E$ has the structure 
of a fake wedge on $K$ and $C$.  Note that 
$\pi_0(\cal D(K,C))$ is the set of equivalence classes of fake wedges on $K$ and $C$.

To formulate our main conjecture, suppose $X$ is a based space.
We let $X^{[j]}$ be the $j$-fold smash product of $X$ and we
let $jX$ denote the wedge of $j$ copies of $X$.
For $n \ge 1$, let $p_n(x,y)$ be the polynomial given by 
\[
(x+ y)^n - x^n - y^n = \sum_{i = 1}^{n-1} \binom{n}{i} x^i y^{n-i}\, .
\]
By analogy, let $p_n(K,C)$ be the space
\[
\bigvee_{i = 1}^{n-1} \binom{n}{i} K^{[i]} \smsh C^{[n-i]}\, .
\]
Then $p_n(K,C)$ is a retract of $(K\vee C)^{[n]}$, and the 
action of the symmetric group $\Sigma_n$ 
leaves the subspace $p_n(K,C)$ invariant. Hence $p_n(K,C)$ comes
equipped with an action of $\Sigma_n$.

We let $\cal W_n$ denote the $n$-th coefficient spectrum of the 
identity functor from based spaces to based spaces in the sense of the calculus  
of homotopy functors \cite[\S8]{Goodwillie3}. In particular, $\cal W_n$ is a spectrum with $\Sigma_n$-action
whose underlying homotopy type is a wedge of $(n-1)!$ copies of the $(1-n)$-sphere spectrum.
We will be considering $\cal W_n \smsh_{h\Sigma_n} p_n(K,C)$,
the homotopy orbit spectrum 
of  $\Sigma_n$ acting diagonally on the smash product $\cal W_n \smsh p_n(K,C)$.
If $E$ is a spectrum and $X$ is a based space, we let 
\[
F^{\text{\rm st}}(X,E) := \Omega^\infty F(X,E)
\]
be the zeroth space of the function spectrum $F(X,E)$. This is the spectrum
whose $j$-th space is the function space of based maps $X\to E_j$. As usual, if necessary
we will replace $E$ by an $\Omega$-spectrum  to insure this delivers a well-defined
homotopy type. When $E = \Sigma^\infty Y$ is a suspension spectrum, we abuse
notation and write $F^{\text{\rm st}}(X,Y)$ in place 
of $F^{\text{\rm st}}(X,\Sigma^\infty Y)$; this is the function space of 
{\it stable maps} from $X$ to $Y$.

\begin{hypo} \label{hypo:rho} Assume that $K$ is $r$-connected and $C$ is $s$-connected with $r,s\ge 1$.
Assume $K$ is homotopy equivalent to a CW complex of dimension $\le k$ and
that $C$ is homotopy equivalent to a CW complex of dimension $\le c$.
Set 
\[
\rho_n =  \min(r + ns,nr+s) + 2  - \max(k,c) \, .
\]
\end{hypo}

\begin{bigconj} \label{bigthm:conjecture1} There is a tower of based spaces
$$
\cdots \to \cal D_n(K,C) \to \cdots \to \cal D_1(K,C) 
$$
and compatible maps $\cal D(K,C) \to \cal D_n(K,C)$ such that
\begin{itemize}
\item $\cal D_1(K,C)$ is a point; 
\item the map $\cal D(K,C) \to \cal D_n(K,C)$ is $\rho_n$-connected. In particular,
the map \[\cal D(K,C) \to \holim_{n} \cal D_n(K,C)\] is a weak equivalence;
\item Each map of the tower sits in a homotopy fiber sequence
\[
\cal D_n(K,C) \to \cal D_{n-1}(K,C) \to 
F^{\text{\rm st}}(K \vee C,\Sigma^2 \cal W_n \smsh_{h\Sigma_n} p_n(K,C))\, .
\]
\end{itemize}
\end{bigconj}

As evidence for the above conjecture we have

\begin{bigthm} \label{bigthm:looped-conjecture1}
Conjecture \ref{bigthm:conjecture1} holds after looping once, i.e., there
is a tower of based spaces
$$
\cdots \to \delta_n(K,C) \to \cdots \to \delta_1(K,C) 
$$
and compatible maps $\Omega\cal D(K,C) \to \delta_n(K,C)$ such that
\begin{itemize}
\item $\delta_1(K,C)$ is a point; 
\item the map $\Omega \cal D(K,C) \to \delta_n(K,C)$ is $(\rho_n-1)$-connected. In particular,
the map $\Omega \cal D(K,C) \to \holim_{n} \delta_n(K,C)$ is a weak equivalence;
\item Each map of the tower sits in a homotopy fiber sequence
\[
\delta_n(K,C) \to \delta_{n-1}(K,C) \to 
F^{\text{\rm st}}(K \vee C,\Sigma \cal W_n \smsh_{h\Sigma_n} p_n(K,C))\, .
\]
\end{itemize}
\end{bigthm}

\begin{rem} By \cite[th.~A(2)]{Klein_pushout} there is a function
\[
\pi_0(\cal D(K,C)) \to \{K \vee C,K\smsh C\}\, ,
\]
where the target is the given by homotopy classes
of stable maps $K \vee C \to K\smsh C$.
It was shown that this function  is surjective if $k,c\le 3\min(r,s)$
and a bijection if $k,c \le 3\min(r,s)-1$. Based on what we
do here, it seems plausible to us that these inequalities can be
improved to $\rho_2 \ge 0$ and $\rho_2 \ge 1$ respectively.

The second coefficient spectrum $\cal W_2$ of 
the identity functor is the $(-1)$-sphere with
trivial $\Sigma_2$-action. Furthermore, $p_2(K,C) = 2K\smsh C$. Hence
$\Sigma\cal W_2 \smsh_{h\Sigma_2} p_2(K,C) \simeq K\smsh C$, so
\cite[th.~A(2)]{Klein_pushout} is really a partial verification of Conjecture 
\ref{bigthm:conjecture1} on the level of path components.
\end{rem}

\begin{rem}\label{rem:function-space-splitting} The stable function space that
appears in Theorem \ref{bigthm:looped-conjecture1}
can be simplified. Greg Arone has explained to us that
for $i,n-i \ge 1$, the spectrum $\cal W_n$ splits
 $(\Sigma_i \times \Sigma_{n-i})$-equivariantly
as a wedge of spectra of the form
\[
(\Sigma_i \times \Sigma_{n-i})_+ \smsh_{\Sigma_d}
 \Sigma^{d-n}\cal W_d\, ,
\]
where $d$ ranges over the common divisors of $i$ and $n-i$. The  number of times that the displayed summand occurs is a certain number $B(i/d, (n-i)/d)$. This is the number of basic products in the free Lie algebra on two generators $x_1, x_2$, involving $i/d$ copies of $x_1$ and $(n-i)/d$ copies of $x_2$ (for details, see \cite[th.~0.1]{AK}).

Equivariantly, the summand of $p_n(K,C)$ given by $\binom{n}{i}K^{[i]}\smsh C^{[n-i]}$
may also be rewritten as 
\[
(\Sigma_n)_+ \smsh_{\Sigma_i \times \Sigma_{n-i}}(K^{[i]}\smsh C^{[n-i]})\, .
\]
These two observations may be combined to show that the spectrum 
$\Sigma\cal W_n \smsh_{h\Sigma_n} p_n(K,C)$ splits into a wedge
of spectra of the form
\[ 
\Sigma^{d-n+1}\cal W_d \smsh_{h\Sigma_d} (K^{[i]}\smsh C^{[n-i]})\, .
\]
It follows that the stable function space $F^{\text{\rm st}}(K \vee C,\Sigma \cal W_n \smsh_{h\Sigma_n} p_n(K,C))$ splits into a product of terms of the form
\begin{equation}\label{eqn:simplified}
F^{\text{\rm st}}(K \vee C,\Sigma^{d-n+1} \cal W_d \smsh_{h\Sigma_d} (K^{[i]} \smsh C^{[n-i]})).
\end{equation}
In special cases we hope to
be able to use this decomposition to obtain numerical information. For example, when $K$ and $C$ are spheres
the rational homotopy type of the stable function
space \eqref{eqn:simplified} is contractible if $d >2$ (and is also contractible for $d=2$ in some cases; cf.\ \cite{Arone-Mahowald}).\footnote{We are indebted to the referee for pointing this out.} More generally, for any $K$ and $C$, the 
rational type of the stable function space
 \eqref{eqn:simplified} may be expressed in terms of the submodule ${\cal Lie}(d)$
of the free Lie algebra ${\cal Lie}(x_1,\dots,x_d)$ over the rational
numbers that is given by all bracket monomials containing $x_i$ exactly once (cf.\ \cite{AK}).
We plan to pursue this matter in another paper.
\end{rem}


\subsection*{Stabilization} Define a functor $D(K,C) \to D(\Sigma K,C)$ by
mapping an object $A$ to $\Sigma_C A$ its fiberwise suspension over $C$. 
$$
\Sigma_C A \,\, := \hocolim (C \leftarrow A \to C)
$$
where the homotopy colimit is taken in the category of based spaces and is given by
the reduced double mapping cylinder $C\times 0 \cup  A \smsh (I_+) \cup C\times 1$. 
There are evident maps from the latter to $\Sigma K = *\times 0 \cup  K \smsh (I_+) \cup *\times 1$
and to $C$, so the assignment $A\mapsto \Sigma_C A$ is  a functor. 

Similarly, one has a functor $D(K,C) \to D(K,\Sigma C)$ which is defined by $A\mapsto \Sigma_K A$. Since homotopy colimits commute, we get a diagram of spaces
$$
\xymatrix{
D(K,C) \ar[r] \ar[d] & D(\Sigma K,C) \ar[d] \\
D(K,\Sigma C) \ar[r] &  D(\Sigma K,\Sigma C)
}
$$
which commutes up to canonical isomorphism.

\begin{bigconj} \label{bigthm:conjecture2} The diagram of classifying spaces
$$
\xymatrix{
\cal D(K,C) \ar[r] \ar[d] & \cal D(\Sigma K,C) \ar[d] \\
\cal D(K,\Sigma C) \ar[r] &  \cal D(\Sigma K,\Sigma C)
}
$$
is $\rho_2$-cartesian, where $\rho_2 = \min(r+2s,2r+s) + 2 - \max(k,c)$.
\end{bigconj}

\begin{rem} This conjecture is in part motivated by
the second author's Ph.D.\ thesis  
where a similar result is proved for embedding spaces 
(cf.\ \cite{JP1},\cite{JP2}).
\end{rem}

As evidence for Conjecture \ref{bigthm:conjecture2}, we have

\begin{bigthm} \label{bigthm:looped-conjecture2} 
Conjecture \ref{bigthm:conjecture2} holds after looping once, i.e.,
the diagram
$$
\xymatrix{
\Omega\cal D(K,C) \ar[r] \ar[d] & \Omega \cal D(\Sigma K,C) \ar[d] \\
\Omega \cal D(K,\Sigma C) \ar[r] &  \Omega\cal D(\Sigma K,\Sigma C)
}
$$
is $(\rho_2-1)$-cartesian.
\end{bigthm}

\begin{convent} We work in the Quillen model category of compactly generated spaces.
Function spaces  are topologized using the compact-open topology. A non-empty space $X$ is $r$-connected
if $\pi_j(X)$ is trivial for $j \le r$. The empty space
is $(-2)$-connected and every non-empty space is $(-1)$-connected.
A map $X \to Y$ of spaces with $Y$ non-empty
is $r$-connected if every one of its homotopy fibers is $(r-1)$-connected.

The equivariant spectra appearing in this paper
are {\it naive:}  a spectrum with $G$-action $E$ is a collection of based $G$-spaces
$\{E_n\}_{n\ge 0}$ equipped with equivariant maps $\Sigma E_n \to E_{n+1}$, where
$G$ acts trivially on the suspension coordinate. Naive $G$-spectra have associated
 homotopy orbits $E_{hG}$ and homotopy fixed points $E^{hG}$. 
 
The category $D(K,C)$ is not small, so in order to 
 take its classifying space one should replace it by an equivalent
 small category. One
 way to do this is to work in a Grothendieck universe. 
 We choose not to address this matter any further.
 \end{convent}

\section{Proof of Theorem \ref{fake_wedge_from_embedding}}

We wish to show that there is no map $A\to K \vee C$ which covers the identity map of $K\times C$
up to homotopy. It will be sufficient to show that the composition
$$
A\to K \times C \to K\smsh C
$$
is non-trivial in reduced singular cohomology.
\medskip

\noindent {\it Case (1).} Suppose that  
$E\to M$ admits a section $\zeta\: M \to E$
such that the composite $M \to E \to C$ is null-homotopic.
The total space of the unit normal sphere bundle of $M$  in $S^n$ is identified with $E$.
Let  $\zeta\: M \to E$ be a section; we can assume by transversality that the image of 
$\zeta$ is contained in $A\subset E$.
We henceforth  identify $M$ with its image under $\zeta$.

Note that $E/M$ is identified with the Thom space
$M^{\xi}$, where $\xi$ denotes bundle given by taking the quotient of
the normal bundle of $M$ by the trivial line bundle defined by the section $\zeta$. 
By choosing
a null homotopy for the composite $M \to A \to C$, we obtain a
map $A \cup CM \to C$, where $A \cup CM$ is the mapping cone of
the inclusion $M \to A$. Identifying $A \cup CM$ with $A/M$, we obtain
a map $A/M \to C$, and it is 
not difficult to see that this map is a homology isomorphism. 
The codimension $\ge 3$ hypothesis
guarantees that $A/M$ and $C$ are $1$-connected. It follows that  $A/M \to C$
is a homotopy equivalence.

The quotient map $q\: A\to A/M$ is injective on cohomology in all degrees, since the inclusion
$M \to A$ is a coretraction. 
Consider the reduced diagonal map
\[
(d)\: A/M @>>> M\smsh A/M
\]
(this is the map of quotients induced by the map of pairs
$(A,M) \to (M\times A,M\times M \cup \ast \times A)$
defined by $x\mapsto (p(x),x)$; here $\ast\in M$ is a choice of basepoint).
With respect to our identifications, the composite 
\[
A@> q >> A/M @> (d) >> M\smsh A/M
\]
is identified up to homotopy with the map $A \to K\smsh C$.
Consequently, it will suffice to prove that the map $(d)$ is non-trivial on 
reduced singular cohomology.

Consider the commutative diagram
$$
\xymatrix{
       E/M \ar[r]^<<<<{(b)} & M_+ \smsh E/M  \ar[dr]^{(c)} & \\
                  &               &     M \smsh E/M \\
       A/M  \ar[r]_<<<<{(d)} \ar[uu]^{(a)}     & M\smsh A/M   \ar[ur]_{(e)} &
}
$$  
where the maps $(a)$ and $(e)$ are inclusions, $(b)$ and $(d)$ are reduced diagonals,   
and $(c)$ is given by sending the additional basepoint to the basepoint of $M$. 
If  $U\in H^{n-m}(E/M)$ is the Thom class, where $m = \dim M$, and
$x \in H^*(M_+)$,  then
the operation $x \mapsto (b)^*(x\times U) = x\cup U$ defines the Thom isomorphism
(where $x \times U$ denotes the reduced external product of $x$ and $U$).
The map $(a)$ is $(n-2)$-connected, so the
Thom class pushes forward to a non-trivial cohomology class $U'\in H^{n-m}(A/M)$.

Let $z\in H^k(M)$ be any non-trivial cohomology class such that $1 < k < m$ (this is 
guaranteed by the assumption that $M$ isn't a homology sphere). 
Then $z\cup U \in  H^{n-m+k}(E/M)$ is non-trivial.  Consequently, so is
$z\cup U' \in H^{n-m+k}(A/M)$, by commutativity of the diagram. 
But $\cup U'$ is the effect of applying $(d)^*$ to the class represented by 
the reduced external product of $z$ and $U'$.
If follows that 
$(d)$ is nontrivial in reduced singular cohomology. 
This completes Case (1).
\medskip

\noindent {\it Case (2): the general case.} We argue by contradiction. Assume
that there is a homotopy equivalence $A \simeq K \vee C$ covering the identity map
of $K\times C$ up to homotopy. We can also think of $A$ as a trivial  wedge on $M$ and
$C$, using the homotopy equivalence $M\simeq K$.  

Think of $M$ as a submanifold of $S^{n+1}$ (using the inclusion 
$S^n \subset S^{n+1}$). Then the complement of $M$ is identified with $\Sigma C$.
Let $E'$ be the boundary of a tubular neighborhood of $M$ in $S^{n+1}$. Then
there is a preferred section $M \to E'$ such that the composite $M \to E' \to \Sigma C$
is null homotopic. To see this, note that $E'$ is a double mapping cylinder of $M$
with itself along $E$ and $\Sigma C$ is a double mapping cylinder of a point with itself along $C$; the map $E' \to \Sigma C$ takes each copy of $M$ to a point.

Hence we may apply Case (1) to obtain in this way a non-trivial 
fake wedge $A'$ on $K$ and $\Sigma C$,
where $A'$ is the effect of removing a point from $E'$. 
If $A$ were trivial then $A'$ would also be trivial (this is because $A'$ is obtained
from $A$ by taking the fiberwise suspension $A$ over $K$.) This gives the contradiction.

\section{Proof of Theorem \ref{bigthm:looped-conjecture1} } 

\subsection*{The moduli space of suspension spectra} 
For an $\Omega$-spectrum $E$ we recall from \cite{Klein_susp} the moduli space
of suspension spectra $\frak M_E$ which was defined to be the realization
of the category $C_E$ whose objects are pairs $(Y,h)$ in which $Y$ is a based
space and $h\: \Sigma^\infty Y \to E$ is a weak (homotopy) equivalence. A morphism
$(Y,h) \to (Y',h')$ is a map of based spaces $f\: Y \to Y'$ such that 
$h'\circ \Sigma^\infty f = h$. Assume in what follows that 
$C_E$ is non-empty.
Let $y = (Y,h)$ be an object of $C_E$. This gives $\frak M_E$ a basepoint.

Assume that $E$ is $r$-connected and has the homotopy type of a CW spectrum
of dimension $\le d$. We will also assume $r \ge 1$.
In \cite[\S6]{Klein_susp} we produced a tower of principal fibrations of based spaces
\[
\cdots \to T_2(y) \to T_1(y) 
\]
and compatible maps $\Omega \frak M_E \to T_j(y)$
such that 
\begin{itemize}
\item $T_1(y)$ is contractible,
\item The map  
$\Omega\frak M_E \to T_j(y)$ is $((j+1)r+1-d)$-connected. In particular, 
the map $\Omega\frak M_E \to \holim_j T_j(y)$ is a weak equivalence.
\item For $j > 1$ there is a homotopy fiber sequence
\begin{equation} \label{eq:principal}
T_j(y) \to T_{j-1}(y) \to 
F^{\text{\rm st}}(Y, \Sigma \cal W_j \smsh_{h\Sigma_j} Y^{[j]})\, .
\end{equation}
\end{itemize}
We will reproduce the homotopy fiber sequence \eqref{eq:principal}
below.

The space $T_j(y)$ is defined to be the space of all lifts filling
in the dotted line in the following diagram
\[
\xymatrix{
&  P_j(Y) \ar[d]\\
Y \ar[r] \ar@{-->}[ur] & P_1(Y)\simeq Q(Y)  \, ,
}
\]
where $P_j(Y)$ is the $j$-th stage of the Goodwillie tower of the identity functor.
The map $T_j(y) \to T_{j-1}(y)$ is obtained
by sending a lift $Y\to P_j(Y)$ to the composite  $Y \to P_j(Y) \to P_{j-1}(Y)$.
The  fiber sequence \eqref{eq:principal} 
is  a consequence of  Goodwillie's observation \cite[lem.~2.2]{Goodwillie3} 
 that $P_j(Y) \to P_{j-1}(Y)$ is a principal
fibration that is classified by a map
$P_{j-1}(Y) \to \Omega^\infty (\Sigma \cal W_j \smsh_{h\Sigma_j} Y^{[j]})$.

We also showed that $\Omega \frak M_E$ is identified up to weak equivalence with
the space of lifts $T_\infty(y)$ of  the diagram
\[
\xymatrix{
&  \holim_j P_j(Y) \simeq Y\ar[d]\\
Y \ar[r] \ar@{-->}[ur] & P_1(Y)  
}
\]
(this uses \cite[2.2.5]{Waldhausen}, or alternatively \cite[\S2]{DK}).

\begin{rem} \label{model} We digress to explain  how \cite[2.2.5]{Waldhausen}, \cite[\S2]{DK} is
 used to give a model for $\Omega \mathfrak M_E$. 
Suppose $\cal C$ denotes a simplicial model category.
Let $c \in \cal C$ be a fibrant and cofibrant object. Consider the category $h\cal C_{(c)}$
consisting of objects of $\cal C$ which are weak equivalent to $C$, where a morphism is 
a weak equivalence. Then the realization of this category,
$|h\cal C_{(c)}|$, coincides with the classifying space $BG(c)$, where $G(c)$ is the
simplicial monoid of homotopy automorphisms of $c$. 

A special case of interest to us occurs when $CÊ= \text{Top}_{/B}$, the simplicial model
category of spaces over $B$, where a morphism $X\to Y$ is a weak equivalence if and only
if it is a weak homotopy equivalence of underlying spaces. In this case $X$ is cofibrant if
and only if its underlying space is a retract of a cell complex and $X$ is fibrant if the structure map
$X\to B$ is a fibration. If $X$ is fibrant and cofibrant, it follows that $|h\text{Top}_{/B,(X)}|$ is a model for the classifying space for the homotopy self-equivalences of $X$ covering the identity of $B$.

In the case appearing in the proof, we may take $B = P_1(Y)$.
Let $Y' \to P_1(Y)$ be a fibrant and cofibrant replacement. Then any endomorphism $Y' \to Y'$ 
is a homology equivalence of underlying spaces since the composite $Y' \to Y' \to Q(Y')$ is
adjoint to the identity map of the suspension spectrum $\Sigma^\infty Y'$.
Furthermore, such endomorphisms are identified with the space of lifts of the map $Y \to P_1(Y)$ through $Y$.
 In particular, when 
$Y$ is $1$-connected, $\Omega\text{Top}_{/P_1(Y),(Y')}$ is identified up to homotopy
with the space of such lifts. Even when $Y$ is not $1$-connected we can still obtain such 
an identification by replacing $Y$ with its plus construction. 
For the details in this case see \cite[prop.~2.3]{Klein_susp}.
\end{rem}

In what follows we set $E = \Sigma^\infty K \times \Sigma^\infty C$.
Define a functor
\[
D(K,C) \to C_E
\]
by $A \mapsto (A,h)$, where the weak equivalence
$h \: \Sigma^\infty A \to E$ is defined by applying $\Sigma^\infty$ to
the structure maps $A\to K$ and $A\to C$.
The space  $K\vee C$ together with the evident
weak equivalence $\Sigma^\infty(K\vee C) \simeq E$
 equips $C_E$ with a basepoint $y$.  

To simplify notation in what follows, we make
an auxiliary definition.

\begin{defn} Suppose $f\:X \to B$ and $g\: Y\to B$ are maps of spaces.
By slight abuse of notation, we let
\[
\xymatrix{
&  X \ar[d]^f\\
Y \ar[r]_g \ar@{-->}[ur] & B  
}
\]
indicate the space of lifts obtained by converting $f$ into a fibration
$E\to B$, replacing $f$ by this fibration and then taking
lifts of the latter diagram.
Equivalently, this is the space of pairs consisting of a map $h'\: Y \to X$
and a choice of commuting homotopy from $f\circ h'$ to $g$.
\end{defn}

Let $\epsilon(K,C)$ be the space 
\[
\xymatrix{
&  K\times C\ar[d]\\
K\vee C \ar[r] \ar@{-->}[ur] & Q(K)\times Q(C)  \, .
}
\]

\begin{lem} \label{sequence} There is a homotopy fiber sequence
\[
\Omega \cal D(K,C) \to \Omega \frak M_E \to 
\epsilon(K,C).
\]
\end{lem}

\begin{proof} We can identify $\Omega \cal D(K,C)$ up to weak equivalence
with the space
\[
\xymatrix{
& K\vee C \ar[d]\\
K\vee C \ar[r] \ar@{-->}[ur] & K\times C  \, ,
}
\]
(cf.\ Remark \ref{model}; we
are also using the observation that any such lift $K\vee C \to K\vee C$ 
appearing above is automatically a weak equivalence). Similarly, $\Omega \frak M_E$ is identified with the space
\[
\xymatrix{
&  K\vee C\ar[d]\\
K\vee C \ar[r] \ar@{-->}[ur] & Q(K)\times Q(C)  \, ,
}
\]
(cf.\ \cite[\S5]{Klein_susp}).

The proof is completed by noticing that
the two lifting spaces sit in a fibration sequence in which the base
space is given by $\epsilon(K,C)$.
\end{proof}

\begin{defn} Let $\epsilon_j(K,C)$ be the space of lifts
\[
\xymatrix{
&  P_j(K) \times P_j(C) \ar[d]\\
K\vee C \ar[r] \ar@{-->}[ur] & Q(K)\times Q(C)  \, .
}
\]
\end{defn}

\begin{proof}[Proof of Theorem \ref{bigthm:looped-conjecture1}]
By naturality, we have an evident map
\[
T_j(y) \to \epsilon_j(K,C)\, .
 \]
Define $\delta_j(K,C)$ to be the homotopy fiber of this map.
Then we have a homotopy fiber sequence
 \[
\delta_j(K,C) \to  T_j(y) \to \epsilon_j(K,C)\, .
 \]
More concretely, $\delta_j(K,C)$ may be identified with the space of lifts
\[
\xymatrix{
&  P_j(K\vee C)  \ar[d]\\
K\vee C \ar[r] \ar@{-->}[ur] & P_j(K) \times P_j(C)  \, .
}
\]
By properties of the Goodwillie tower of the identity functor,
the map  $\holim_j \delta_j(K,C) \to \delta_j(K,C)$ is $((j+1)r+1-d)$-connected,
where $r$ is the connectivity of $K \vee C$ and $d = \max(k,c)$ is the homotopy dimension of 
$K \vee C$. Furthermore, $\holim_j \delta_j(K,C)$ is identified with $\Omega \cal D(K,C)$.

By comparing this sequence for the corresponding one for
 $j-1$, we obtain a commutative diagram whose rows and 
 columns form homotopy fiber sequences
 \[
\resizebox{6.5in}{!}
{
 \xymatrix{
 \delta_j(K,C) \ar[r]\ar[d] &  T_j(y) \ar[r]\ar[d] &  \epsilon_j(K,C)\ar[d] \\
 \delta_{j-1}(K,C) \ar[r]\ar[d] &  T_{j-1}(y) \ar[r]\ar[d] &  \epsilon_{j-1}(K,C)\ar[d] \\
 ? \ar[r] & F^{\text{\rm st}}
 (K \vee C,\Sigma \cal W_j \smsh_{h\Sigma_j} (K\vee C)^{[j]})
 \ar[r]_{f} &  
 F^{\text{\rm st}}(K \vee C,\Sigma \cal W_j \smsh_{h\Sigma_j} (K^{[j]} \vee  C^{[j]}))
 }  
 }
 \]
where the space $?$ is given by the homotopy fiber of the map labelled $f$
that is induced by the projection map
$(K\vee C)^{[j]} \to K^{[j]} \vee C^{[j]}$. It follows that 
$?$ is identified with  the stable function space
$F^{\text{\rm st}}(K \vee C,\Sigma \cal W_j \smsh_{h\Sigma_j} p_j(K,C))$.
This establishes the homotopy fiber sequence 
\begin{equation} \label{eq:delta-sequence}
\delta_j(K,C) \to  \delta_{j-1}(K,C) \to F^{\text{\rm st}}(K \vee C,\Sigma \cal W_j \smsh_{h\Sigma_j} p_j(K,C))\, .
\end{equation}
A straightforward calculation we  omit shows 
$F^{\text{\rm st}}(K \vee C,\Sigma \cal W_{j+1} \smsh_{h\Sigma_{j+1}} p_{j+1}(K,C))$
to be $(\rho_{j}-1)$-connected. In particular, the map $\delta_{j+1}(K,C) \to  \delta_{j}(K,C)$
is $(\rho_{j}-1)$-connected. As $\rho_j$ is an increasing function of $j$, it follows
that the map $\Omega \cal D(K,C) \simeq \holim_i \delta_i(K,C) \to \delta_j(K,C)$ is $(\rho_j-1)$-connected.
\end{proof}

\section{Proof of Theorem \ref{bigthm:looped-conjecture2}}

To prove  Theorem \ref{bigthm:looped-conjecture2}, by induction it will
be enough to
show that for each $j \ge 1$ a certain commutative diagram
\begin{equation} \label{eq:delta-diagram}
\xymatrix{
\delta_j(K,C) \ar[r] \ar[d] & \delta_j(K,\Sigma C) \ar[d] \\
\delta_j(\Sigma K,C) \ar[r] & \delta_j(\Sigma K,\Sigma C)
}
\end{equation}
is $(\rho_2 -1)$-cartesian (see below for 
a description of the maps of this diagram). 

It will turn out that these diagrams are functorial in the index
$j$, so by the homotopy fiber sequence \eqref{eq:delta-sequence},  we are reduced to 
showing that the associated diagram
\[
\resizebox{6in}{!}
{\xymatrix{
F^{\text{\rm st}}(K \vee C,\Sigma \cal W_j \smsh_{h\Sigma_j} p_j(K,C)) 
\ar[r] \ar[d] & F^{\text{\rm st}}(K \vee \Sigma C,\Sigma \cal W_j \smsh_{h\Sigma_j} p_j(K,\Sigma C)) \ar[d] \\
F^{\text{\rm st}}(\Sigma K \vee C,\Sigma \cal W_j \smsh_{h\Sigma_j} p_j(\Sigma K,C))
\ar[r] & 
F^{\text{\rm st}}(\Sigma K \vee \Sigma C,\Sigma \cal W_j \smsh_{h\Sigma_j} p_j(\Sigma K,\Sigma C))
}
}
\]
is  $\rho_2$-cartesian (again, we will describe the maps
of the diagram below). 

The proof will use the projection map
\[
p_j(K,C) \to j ((K^{[j-1]} \smsh C)  \vee (K \smsh C^{[j-1]})) \, ,
\]
for $j \ge 3$.
When $j=2$, our convention will be that the target of the projection map is 
$2K\smsh C$. The projection
 map is $(2(r+s) +4)$-connected for all $j \ge 2$.

In what follows we set
\[
q_j(K,C) =  j ((K^{[j-1]} \smsh C)  \vee (K \smsh C^{[j-1]}))
\]
for $j \ge 3$ and when $k = 2$ we set $q_2(K,C) = 2K\smsh C$.
It is elementary to show that the induced map
\[
F^{\text{\rm st}}(K \vee C,\Sigma \cal W_j \smsh_{h\Sigma_j} p_j(K,C)) 
\to F^{\text{\rm st}}(K \vee C,\Sigma \cal W_j \smsh_{h\Sigma_j} 
q_j(K,C))
\]
is  $(2(r+s) + 2 -d)$-connected, $d = \max(k,c)$.
Note that the action of $\Sigma_j$  on $q_j(K,C)$ is given by
inducing up the evident $\Sigma_{j-1}$-actions on the summands $K^{[j-1]} \smsh C$
and $K \smsh C^{[j-1]}$.

Consequently, we are reduced to checking that
the diagram
\begin{equation} 
\label{eq:qk}
\resizebox{6in}{!}
{
\xymatrix{
F^{\text{\rm st}}(K \vee C,\Sigma \cal W_j \smsh_{h\Sigma_j} q_j(K,C)) 
\ar[r] \ar[d] & F^{\text{\rm st}}(K \vee \Sigma C,\Sigma \cal W_j \smsh_{h\Sigma_j} q_j(K,\Sigma C)) \ar[d] \\
F^{\text{\rm st}}(\Sigma K \vee C,\Sigma \cal W_j \smsh_{h\Sigma_j} q_j(\Sigma K,C))
\ar[r] & 
F^{\text{\rm st}}(\Sigma K \vee \Sigma C,\Sigma \cal W_j \smsh_{h\Sigma_j} q_j(\Sigma K,\Sigma C))
}
}
\end{equation}
is $\rho_2$-cartesian.

Set 
\[
u_j(K,C) = K \smsh C^{[j-1]} \quad \text{ and }\quad  v_j(K,C) = K^{[j-1]} \smsh C\, .
\] 
Then $q_j(K,C) = j(u_j(K,C) \vee v_j(K,C))$.
The restriction to $\Sigma_{j-1}$ of the action of $\Sigma_j$ on $\cal W_j$
is identified with the permutation action on $(\Sigma_{j-1})_+\smsh S^{1-j}$ 
(cf.\ Remark \ref{rem:function-space-splitting}). Similarly,
the action of $\Sigma_j$ on $q_j(K,C)$ is given by inducing up from 
the permutation action of $\Sigma_{j-1}$ on $u_j(K,C) \vee v_j(K,C)$. Hence,
\[
F^{\text{\rm st}}(K \vee C,\Sigma \cal W_j \smsh_{h\Sigma_j} q_j(K,C))
\simeq  F^{\text{\rm st}}(K \vee C,\Sigma^{2-j}(u_j(K,C) \vee v_j(K,C)))\, .
\]
Observe that for $j\ge 3$, the inclusion map 
\begin{equation*}
\resizebox{1.0\hsize}{!}{$F^{\text{\rm st}}(K,\Sigma^{2-j} u_j(K,C))  \times F^{\text{\rm st}}(C,\Sigma^{2-j} v_j(K,C)) 
\to 
F^{\text{\rm st}}(K \vee C,\Sigma^{2-j}(u_j(K,C) \vee v_j(K,C)))$}
\end{equation*}
is $\rho_2$-connected.

Then the above reduces us to showing that a certain commutative diagram
\begin{equation} 
\label{eq:uvj}
\resizebox{6in}{!}
{
\xymatrix{
F^{\text{\rm st}}(K,\Sigma^{2-j} u_j(K,C)) {\times} F^{\text{\rm st}}(C,\Sigma^{2-j}v_j(K,C))
\ar[r]\ar[d] & F^{\text{\rm st}}(K,\Sigma^{2-j}u_j(K,\Sigma C)) {\times} F^{\text{\rm st}}(\Sigma C,\Sigma^{2-j} v_j(K,\Sigma C))  \ar[d] \\
F^{\text{\rm st}}(\Sigma K,\Sigma^{2-j} u_j(\Sigma K,C)) {\times} F^{\text{\rm st}}(C,\Sigma^{2-j} v_j(\Sigma K,C))
\ar[r] & F^{\text{\rm st}}(\Sigma K,\Sigma^{2-j} u_j(\Sigma K,\Sigma C)) {\times} F^{\text{\rm st}}(\Sigma C,\Sigma^{2-j}v_j(\Sigma K,\Sigma C)) 
}
}
\end{equation}
is $\rho_2$-cartesian. We will explain the maps of this diagram below, and we will
show that the diagram is in fact $\infty$-cartesian. This will complete the proof.
 
For $j \ge 2$, a point of $F^{\text{\rm st}}(K,\Sigma^{2-j} u_j(K,C)) \times F^{\text{\rm st}}(C,\Sigma^{2-j} v_j(K,C))$  is given by
a pair of stable maps 
\[
a\: K \to \Sigma^{2-j} K\smsh C^{[j-1]} \quad \text{ and } \quad
b\: C \to \Sigma^{2-j} K^{[j-1]} \smsh C\, .
\]
Then we claim that the horizontal arrows  
of the diagram \eqref{eq:uvj} are given by the operation $(a,b)\mapsto (\ast,\Sigma b)$. Similarly, we claim that the vertical arrows are given by
the operation $(a,b) \mapsto (\Sigma a,\ast)$. 
From this description, and the fact that
$F^{\text{st}}(X,Y) = F^{\text{st}}(\Sigma X,\Sigma Y)$,
 it is clear
that diagram \eqref{eq:uvj} is $\infty$-cartesian.

We now fill in some of the details in the above argument. 
We first need to define the stabilization maps appearing
in diagram \eqref{eq:delta-diagram}. Then we
we need to show that the following diagram commutes
\begin{equation}\label{eq:to-check}
\resizebox{6in}{!}
{
\xymatrix{
\Omega \cal D(K,C) \ar[r] \ar[d] & \delta_{j-1}(K,C) \ar[d] \ar[r] & F^{\text{\rm st}}(K,\Sigma^{2-j} u_j(K,C)) \times F^{\text{\rm st}}(C,\Sigma^{2-j} v_j(K,C))\ar[d]  \\
\Omega \cal D(K,\Sigma C) \ar[r] & \delta_{j-1}(K,\Sigma C) \ar[r] & F^{\text{\rm st}}(K,\Sigma^{2-j} u_j(K,\Sigma C)) \times F^{\text{\rm st}}(C,\Sigma^{2-j} v_j(K,\Sigma C))
}
}
\end{equation}
where the left vertical map is the given stabilization map and the right vertical map
is the one given above (i.e., $(a,b) \mapsto (\ast,\Sigma b)$). 
The middle vertical map is a map of diagram \eqref{eq:delta-diagram}
with a change of index from $j$ to $j-1$.
We also need to
show that a diagram analogous to diagram \eqref{eq:to-check}
that involves the other kind of stabilization map 
commutes. We will omit that argument since it is essentially the same.

We now define the map $\delta_j(K,C) \to \delta_j(K,\Sigma C)$.
A point of  $\delta_j(K,C)$ is represented by a map $g\: K\vee C \to P_j(K\vee C)$ 
which projects to the canonical map $K \times C \to  P_j K \times P_j C$ up to a
choice of commuting homotopy, which for reasons of clutter will be ignored
in our description.  We map $g$ to the composite
\begin{equation} \label{eq:sigma-g}
\Sigma_K (K\vee C) \to \Sigma_{P_jK} P_j(K \vee C) \to
P_j(\Sigma_K(K \vee C)) \, ,
\end{equation}
where the second map in the composite arises from applying the functor $P_j$ to the homotopy pushout diagram
\[
\xymatrix{
K \vee C \ar[r] \ar[d] & K \ar[d] \\
K \ar[r] & \Sigma_K(K\vee C)
}
\]
to get  
\begin{equation} \small \label{eq:C-stable}
\text{hocolim}(P_j K \leftarrow P_j(K\vee C) \to P_jK) \to P_j(\Sigma_K(K\vee C))\, .
\end{equation}
Note that $\Sigma_{P_j K}P_j(K\vee C)$ coincides with the domain of the map 
\eqref{eq:C-stable}.
We are also implicitly using the identification 
\[
\Sigma_K(K\vee C) = K \vee \Sigma C
\]
to identify the map \eqref{eq:sigma-g} with an element of $\delta_j(K,\Sigma C)$.
A similar description defines the other stabilization map $\delta_j(K,C) \to 
\delta_j(\Sigma K,C)$
It is now  straightforward to  check that  diagram
\eqref{eq:delta-diagram} is commutative and is compatible with passage
from $j$ to $j-1$. It is also not difficult to check that the left hand square
of diagram \eqref{eq:to-check} is commutative, and we will omit these details. 
It remains to verify the commutativity
of the right hand square of diagram \eqref{eq:to-check}.
This will involve a long and tedious diagram chase and we will be content
providing an outline of the argument.

By construction, the map
\[
\delta_{j-1}(K,C) \to F^{\text{\rm st}}(K,\Sigma^{2-j} u_j(K,C)) \times F^{\text{\rm st}}(C,\Sigma^{2-j}v_j(K,C))
\]
factors through $F^{\text{\rm st}}(K\vee C,\Sigma \cal W_j\smsh_{h\Sigma_j}(K\vee C)^{[j]})$.
We will first describe how the map 
\[
\delta_{j-1}(K,C) \to F^{\text{\rm st}}(K\vee C,\Sigma \cal W_j\smsh_{h\Sigma_j}(K\vee C)^{[j]})
\]
behaves with respect to stabilization by constructing a stabilization map for the target.

Let $L_jX = \Sigma \cal W_j\smsh_{h\Sigma_j}X^{[j]}$.
Let $g\: K \vee C \to P_{j-1}(K\vee C)$ represent a point of $\delta_{j-1}(K,C)$
Then the chain of maps
\[
K\vee C \overset g\to P_{j-1}(K\vee C) \to L_j(K\vee C)
\]
gives rise to a commutative diagram
\[
\xymatrix{
\Sigma_K(K\vee C) \ar[r] \ar@{=}[d] & \Sigma_{P_{j-1}K} P_{j-1}(K\vee C) \ar[r] \ar[d] & 
\Sigma_{L_jK} L_j(K \vee C)\ar[d]  \\
\Sigma_K(K\vee C)  \ar[r] & P_{j-1}(\Sigma_K (K\vee C)) \ar[r] & 
L_j(\Sigma_K(K\vee C))\, ,
}
\]
The bottom left map of this diagram represents the stabilization of $g$ lying in
$\delta_{j-1}(K,\Sigma C)$. The bottom composite represents the image of
this with respect to the map $\delta_{j-1}(K,\Sigma C) \to 
F^{\text{\rm st}}(K\vee \Sigma C,L_j(K\vee \Sigma C))$. The commutativity
of this diagram implies the commutativity of  the diagram 
\[
\xymatrix{
\delta_{j-1}(K,C) \ar[r]\ar[d] & F^{\text{\rm st}}(K\vee C,L_j(K\vee C)) \ar[d]\\
\delta_{j-1}(K,\Sigma C) \ar[r] & F^{\text{\rm st}}(K\vee \Sigma C,L_j(K\vee \Sigma C))
}
\] 
where the vertical maps are stabilization maps and the right vertical map 
is given by sending a stable map $K\vee C \to L_j(K\vee C)$
to the stable composite 
\[
\Sigma_K(K\vee C) \to \Sigma_{L_jK} L_j(K\vee C) \to L_j(\Sigma_K (K\vee C))\, .
\]
Finally, a straightforward 
check that we omit shows 
that the diagram 
\[
\resizebox{5in}{!}
{
\xymatrix{
F^{\text{\rm st}}(K\vee C,L_j(K\vee C)) \ar[r] \ar[d] & 
F^{\text{\rm st}}(K,\Sigma^{2-j} u_j(K,C)) \times F^{\text{\rm st}}(C,\Sigma^{2-j} v_j(K,C)) \ar[d] \\
F^{\text{\rm st}}(K\vee \Sigma C,L_j(K\vee \Sigma  C) \ar[r]  & 
F^{\text{\rm st}}(K,\Sigma^{2-j} u_j(K,\Sigma C)) \times F^{\text{\rm st}}(\Sigma C,\Sigma^{2-j} v_j(K,\Sigma C))
}
}
\]
is commutative, where the top horizontal map is induced by the  
projection $L_j(K\vee C) \to \Sigma^{2-j} u_j(K,C) \times \Sigma^{2-j} u_j(K,C)$
and the right vertical map is given by
$(a,b) \mapsto (\ast,\Sigma b)$.

\section{Appendix: discussions with Arone and Thomason in 1995}

In Bielefeld in early 1995, the first author was involved in
a series of discussions with 
Greg Arone and Bob Thomason about the possibility of developing a
theory of $E_\infty$-coalgebras over the sphere spectrum to
serve as a recognition principle for deciding when a spectrum is weakly
equivalent to
a suspension spectrum. Unfortunately, Bob Thomason
passed away in November, 1995 before
we could  get such a theory up and running.
However, in \cite{Klein_susp} the first author took the first step in this
direction by constructing a theory of coalgebra
spectra in the metastable range (cf.\ below). In \cite{Klein_emb_sphere} we used this theory to give concrete results about embeddings. 
We will now try to outline some aspects of this project because it is relevant to the 
 topic of this paper.

First observe that a based space $X$ gives rise  to a suspension
spectrum $\Sigma^\infty X$. The diagonal map $\Delta\: X \to X \smsh X$ 
induces a diagonal $\Delta\: \Sigma^\infty X \to \Sigma^\infty X\smsh 
\Sigma^\infty X$ which is commutative up to all higher coherences. This gives
$\Sigma^\infty X$ the structure of a coalgebra over the sphere spectrum
(without co-unit) that is coherently homotopy commutative.
Conversely, suppose we could define a category of $E_\infty$-coalgebras over $S^0$.
Then one might hope that
 the functor $X \to \Sigma^\infty X$  induces an equivalence
of associated homotopy categories of $1$-connected objects.

The vague conjectural idea is that whatever an $E_\infty$-coalgebra is, it should amount
in some way to a spectrum admitting a co-action by the $E_\infty$-operad. The filtration
of $E\Sigma_\infty$ by the $E\Sigma_n$ should then somehow correspond to the stages
of the Goodwillie tower of the identity functor 
from spaces to spaces. More precisely, if 
we start with an $E_\infty$-coalgebra structure on a spectrum 
$E$, then there should be a functorially
associated tower of fibrations of
based spaces $\cdots \to P_2E \to P_1E$ in which 
\begin{itemize}
\item $P_1E  = \Omega^\infty E$;
\item For $n \ge 2$, the $n$-th layer $L_nE = \text{fiber}(P_nE \to P_{n-1}E)$ is given by 
the infinite loop space
$\Omega^\infty (\cal W_n \smsh_{h\Sigma n} E^{[n]})$;
\item If  $X = \lim_n P_n E$ then the map $X\to \Omega^\infty E$ is adjoint
to a weak equivalence when $E$ is $1$-connected.
\end{itemize}
Furthermore, the partial tower $P_n E \to \cdots \to P_1 E$ should functorially
only depend on the co-action by the portion of the $E_\infty$-operad that involves
$E\Sigma_j$ for $j \le n$. That is, the partial tower should depend on less
structure: it should only require a choice of coaction on  $E$ of 
the {\it $n$-truncation} of the $E_\infty$-operad in the sense of 
\cite{Fiedorowicz}, \cite[defn.~2.17]{Ching}. We will be somewhat
more precise about this below.
  
First consider the  $n=2$ case. We have a norm sequence
\[
D_2E \to (E\smsh E)^{h\Sigma_2} \to (E\smsh E)^{t\Sigma_2}
\]
in which $D_2E$ is the quadratic construction, $(E\smsh E)^{h\Sigma_2}$
is the homotopy fixed points of $\Sigma_2$ acting on $E\smsh E$ and
$(E\smsh E)^{t\Sigma_2}$ is the Tate construction of $\Sigma_2$ 
acting on $E\smsh E$.
There is also a natural map $q_2: E \to (E\smsh E)^{t\Sigma_2}$. 
We now define a $2$-truncated $E_\infty$-coalgebra
structure on $E$ to be a map $E \to (E\smsh E)^{h\Sigma_2}$ 
which factorizes $q_2$.
What we have just defined is equivalent
to the notion of diagonal given in \cite{Klein_susp} and this
is a good enough notion in the metastable range:
using this lift, we were able to construct a
partial tower $P_2E \to \Omega^\infty E$ and we showed that 
the adjoint map $\Sigma^\infty P_2E \to E$ 
is an equivalence in the metastable range.

The above suggests that an $n$-truncated 
$E_\infty$-coalgebra structure on a spectrum $E$ might have
an inductive definition. Given an $(n-1)$-truncated 
$E_\infty$-coalgebra structure on $E$, we conjecture
that one can associate to it 
a natural map $q_n\: E \to (E^{[n]})^{t\Sigma_n}$
(this is easy to do when $E$ is a suspension spectrum, but we
do not yet know how to do this in general).
Assuming this can be done,
we define an $n$-truncated structure which extends
the given $(n-1)$-truncated structure
by choosing a factorization $E \to (E^{[n]})^{h\Sigma_n}
\to (E^{[n]})^{t\Sigma_n}$. We surmise that this definition
enables one to construct an $n$-stage tower $P_nE \to \cdots \to P_1E = \Omega^\infty E$
such that $\Sigma^\infty P_n E \to E$ is $(nr+r-1)$-connected when $E$ is $r$-connected.

If this idea works, 
then we should be able to study the moduli space of suspension structures
$\frak M_E$ by constructing
a sequence of approximations  $\cdots \to \frak M_{E,n} \to \frak M_{E,n-1}$
in which $\frak M_E = \lim_n \frak M_{E,n}$ and  $\frak M_{E,n}$ 
is a moduli space of $n$-truncated $E_\infty$-structures on $E$ (when 
$\frak M_E$ is non-empty the space
$\frak M_{E,n}$ ought to be a non-connective delooping of 
the space $T_n(y)$ appearing above).
Moreover, there
should be a homotopy fiber sequence
\[
\Omega^\infty F(E,\Sigma \cal W_n\smsh_{h\Sigma_n} E^{[n]}) \to
 \frak M_{E,n} \to \frak M_{E,n-1} \, .
 \]
We could then use  $\frak M_{E,n}$ in place of $T_n(y)$ in the proof
of Theorem \ref{bigthm:looped-conjecture1} to 
obtain a delooping of the proof.

\end{document}